\documentclass[11pt,a4paper,fleqn]{amsart}
\usepackage{a4wide}
\usepackage{amsthm, amsmath, amsfonts, amssymb,graphicx}
\usepackage{mathrsfs,url,color}

\newcommand{\NN}{\mathbb{N}}

\newcommand{\OO}{\mathscr{O}}

\newcounter{dummy}

\newcommand{\sset}{\subset}

\newcommand{\sseteq}{\subseteq}

\DeclareMathOperator{\rv}{rv}
\DeclareMathOperator{\RV}{RV}
\DeclareMathOperator{\VF}{VF}

\DeclareMathOperator{\dcl}{dcl}
\DeclareMathOperator{\acl}{acl}

\DeclareMathOperator{\res}{res}

\DeclareMathOperator{\Th}{Th}
\DeclareMathOperator{\tp}{tp}

\def\hmix{h$^{\mathrm{mix}}$}

\theoremstyle{plain}
\newtheorem{thmintro}{Theorem}
\newtheorem{corintro}[thmintro]{Corollary}
\newtheorem{theorem}{Theorem}[section]
\newtheorem{prop}[theorem]{Proposition}
\newtheorem{lem}[theorem]{Lemma}
\newtheorem{qu}[theorem]{Question}
\newtheorem{quintro}[thmintro]{Question}
\newtheorem*{claim*}{Claim}
\newtheorem*{subclaim*}{Subclaim}
\newtheorem*{claima*}{Claim~A}
\newtheorem*{claimb*}{Claim~B}
\newtheorem*{claimc*}{Claim~C}

 \theoremstyle{definition}
\newtheorem{defn}[theorem]{Definition}
 \newtheorem{exa}[theorem]{Example}
 \newtheorem{rmk}[theorem]{Remark}

\begin{document}

\title[Spherically~complete Hensel minimal fields]{Spherically~complete~models of~Hensel~minimal~valued~fields}
\author{David Bradley-Williams}
\author{Immanuel Halupczok}
\address{Mathematisches Institut der Heinrich-Heine-Universit\"at, Universit\"atsstr. 1, 40225 D\"usseldorf, Germany}

\thanks{The authors would like to thank Franz-Viktor Kuhlmann for helpful comments about valuation theory in mixed characteristic and Martin Hils for asking a question leading us to Remark \ref{r.sph.sat}, as well as the anonymous referee for valuable comments that improved the presentation of our results and for posing Question \ref{qu.not.finram}. Both authors would like to thank the organisers of the \emph{Model Theory, Combinatorics and Valued fields} trimester at the IHP, Paris in 2018, where some of this work was carried out; the participation of D.B-W. in the trimester was partially supported by a HeRA travel grant. Both authors were partially supported by the research training group
\emph{GRK 2240: Algebro-Geometric Methods in Algebra, Arithmetic and Topology}, funded by the DFG. I.H. was additionally partially suppored by the individual research grant No. 426488848, \emph{Archimedische und nicht-archimedische Stratifizierungen h\"oherer Ordnung}, also funded by the DFG. D.B-W. would also like to thank the Fields Institute, Toronto for hospitality and support during the 2021 thematic program
``Trends in Pure and Applied Model Theory'', which provided the ideal environment for various revisions and improvements.}



\begin{abstract}
We prove that Hensel minimal expansions of finitely ramified Henselian valued fields admit spherically complete immediate elementary extensions. More precisely, the version of Hensel minimality we use is $0$-\hmix-minimality (which, in equi-characteristic $0$, amounts to $0$-h-minimality).
\end{abstract}

\maketitle

Geometry tends to be a lot easier over $\mathbb{R}$ than over $\mathbb{Q}$ since $\mathbb{R}$ is complete. Something similar is true in valued fields, where the most useful notion is \emph{spherical completeness}, meaning that every chain of valuative balls has non-empty intersection.
Recall that a valued field is spherically complete if and only if it is \emph{maximal}, i.e., if it has no proper immediate extension. (This is the statement of \cite[Theorem~4]{Kap1942}, where spherical completeness is phrased in terms of pseudo-convergent sequences.)
A classical result of W. Krull ({\cite[Satz~24]{Kru1932}}, see \cite{Gra1956} for a short proof of the required estimate) yields that every valued field has a maximal immediate extension. If $K$ is of equi-characteristic $0$, then this maximal immediate extension is even unique up to valued-field isomorphism (\cite{Kap1942}, Theorem 5 combined with Theorem 6): It is the Hahn field $k((t^{\Gamma}))$, where $k$ is the residue field of $K$ and $\Gamma$ is its value group. In mixed characteristic, uniqueness of maximal immediate extensions is a subtle issue in general (see \cite{Kap1942}, elucidated by \cite{KPR1986}, and a recent update in \cite{BK2017}). However for finitely ramified valued fields of mixed characteristic, it is known that the maximal immediate extension is unique \cite[Corollary 4.29]{vdD14}. Recall that a valued field of finite residue characteristic $p$ is \emph{finitely ramified} if there are only finitely many values in the value group $\Gamma$ between $|p|$ and $|1|=1$, where $|x|$ stands for the valuation of $x$, which we write multiplicatively.

\setcounter{thmintro}{-1}

Spherical completeness is not a first order property;
i.e., if $K'$ is spherically complete and $K \equiv K'$ is elementarily equivalent (in some suitable valued field language), then $K$ does not need to be spherically complete. However, such a $K$ is certainly ``definably spherically complete'', i.e., every \emph{definable} chain of balls in $K$ has non-empty intersection (see Section 1 for a more detailed definition). This raises the following:
\begin{quintro}\label{Q}
If $K$ is definably spherically complete, does it have an (actually) spherically complete elementary extension $L$? Can $L$ additionally be taken to be an immediate extension of $K$?
\end{quintro}

If $K$ has equi-characteristic $0$ and
we work in the pure valued field language
(say, the ring language together with a predicate for the valuation ring), then the answer to Question~\ref{Q} is yes-and-yes: In that case, $K$ is Henselian, and then
Ax--Kochen/Ershov implies that its maximal immediate extension is an elementary extension.
In particular, given a Henselian valued field $K$ of equi-characteristic $0$, a model theorist may assume without loss that $K$ is spherically complete, by passing to an elementary extension. (We use this crucially in \cite{BWH-canstrats}.)

The aim of this note is to prove that Question~\ref{Q} also has a positive answer for certain expansions of the valued field language on $K$, namely
when the structure on $K$ is $0$-h-minimal.
Recall that
$0$-h-minimality is an analogue of o-minimality for valued fields. More precisely, in 
\cite[Definition~1.2.3]{CHR-hmin}, a whole family of such analogues are introduced, called ``Hensel minimality'' in general, where $0$-h-minimality is the weakest of them.
Note that Question~\ref{Q} was already known before to have a positive answer in some (other) cases; see Section~\ref{sec:qu}.

We also obtain a positive answer in mixed characteristic under the assumption that $K$ is finitely ramified. More precisely, in mixed characteristic,
there are several variants of $0$-h-minimality; the one we use is $0$-\hmix-minimality in the sense of \cite[Definition~2.2.1]{CHRV-hmin2}.
Note that among all the variants of Hensel minimality introduced in \cite{CHRV-hmin2} (called $\ell$-h$^{\star}$-minimality for $\ell \in \NN \cup \{\omega\}$ and $\star \in \{\mathrm{mix}, \mathrm{coars}, \mathrm{ecc}\}$), the only ones \emph{not} currently known to imply $0$-\hmix-minimality are $0$-h$^{\mathrm{coars}}$-minimality and $0$-h$^{\mathrm{ecc}}$-minimality.\footnote{For $\ell \ge 1$, $\ell$-h$^\star$-minimality implies $1$-h$^\star$-minimality by definition, which implies $1$-\hmix-minimality by \cite[Theorem~2.2.8]{CHRV-hmin2}, which implies $0$-\hmix-minimality by definition.\refstepcounter{dummy}\label{foo}} Note also that $0$-\hmix-minimality is equivalent to $0$-h-minimality in equi-characteristic $0$.

Assuming finite ramification in the mixed characteristic case seems natural: 
$0$-\hmix-minimality does not imply definable spherical completeness in general (see Example~\ref{exa:not-dsc}),
but it does if we assume finite ramification (see Proposition~\ref{prop:dsc-mix}).

To summarize, our main result is the following:

 \begin{thmintro}\label{t.sphc.ex}
Let $K$ be a characteristic 0 valued field, considered as a structure in a language expanding the valued field language, such that, either
\begin{itemize}
 \item The residue characteristic of $K$ is 0 and $\Th(K)$ is 0-h-minimal; or
 \item $K$ is finitely ramified, of finite residue characteristic, and $\Th(K)$ is $0$-\hmix-minimal.
\end{itemize}
Then $K$ has an elementary extension $L \succ K$ which is immediate and spherically complete (and hence maximal).
\end{thmintro}

By Remark~\ref{r.sph.sat}, passing from $K$ to $L$ preserves saturation, so we also obtain the following:

\begin{corintro}
Given any cardinal $\kappa$, any $K$ as in Theorem~\ref{t.sphc.ex} has a $\kappa$-saturated elementary extension that is moreover spherically complete.
\end{corintro}

\medskip

The notion of $0$-h-minimality is defined by limiting the complexity of definable subsets of the field (in an analogy to o-minimality).
We refer to \cite{CHR-hmin} for details, particularly Definition 1.2.3 and Section 1.2 for the main definitions, and to \cite{CHRV-hmin2}, particularly Section 2, for the mixed characteristic versions. While we do not recall the original definitions of those minimality notions, Lemma~\ref{lem.hmin.equi} below may be taken as a definition of $0$-h-minimality and Lemma~\ref{lem.hmin.mix} may be taken as a definition of $0$-\hmix-minimality.

Many examples of valued fields with $0$-h-minimal and $0$-\hmix-minimal theories are given in \cite[Section~6]{CHR-hmin}. Those in particular include all Henselian valued fields of characteristic $0$ in the valued field language, expansions by various kinds of analytic functions, and the power bounded $T$-convex valued fields from \cite{DL.Tcon1} (see Section~\ref{sec:qu}).

\section{Notations and Assumptions}

We use notation and conventions as in \cite[Section 1.2]{CHR-hmin}:
\begin{itemize}
    \item 
Given a valued field $K$, we write $|x|$ for the valuation of an element $x \in K$, we use multiplicative notation for the value group, which is denoted by $\Gamma^\times_K$, and we set
$\Gamma_K := \Gamma^\times_K \cup \{0\}$, where $0 := |0|$. The valuation ring of $K$ is denoted by $\OO_K$.
\item
Given $a \in K$ and $\lambda \in \Gamma_K$, we write $B_{<\lambda}(a) = \{x \in K : |x-a| < \lambda\}$ for the open ball of radius $\lambda$ around $a$.
\item
Given $\lambda \le |1|$ in $\Gamma^\times_K$, we write $\RV_{\lambda,K} := (K^\times/B_{<\lambda}(1)) \cup \{0\}$ for the corresponding leading term structure (recall that $B_{<\lambda}(1)$ is a subgroup of the multiplicative group $K^\times$) and $\rv_\lambda\colon K \to \RV_{\lambda,K}$ for the canonical map.
In the case $\lambda = |1|$, we will also omit the index $\lambda$, writing $\rv\colon K \to \RV_{K}$.
\item We will freely consider $\Gamma_K$ as an imaginary sort, and also $\RV_{\lambda,K}$, when $\lambda$ is $\emptyset$-definable. (Actually, the only $\lambda$ relevant in this note are of the form $|p^\nu|$, where $p$ is the residue characteristic.)
\end{itemize}

The notions of $0$-h-minimality (from \cite{CHR-hmin}) and $0$-\hmix-minimality (from \cite{CHRV-hmin2}) are defined by imposing conditions on definable subsets of $K$. Instead of recalling those definitions, we cite some direct consequences as Lemmas~\ref{lem.hmin.equi} and \ref{lem.hmin.mix}. Those are simply family versions of the definitions
of $0$-h-minimality and $0$-\hmix-minimality. (The actual definitions are obtained by restricting those lemmas to the case $k = 0$.)

\begin{lem}[{\cite[Proposition~2.6.2]{CHR-hmin}}]\label{lem.hmin.equi}
Suppose that $K$ is a valued field of equi-characteristic $0$ with $0$-h-minimal theory, that $A \subseteq K$ is an arbitrary (parameter) set and that $W \subseteq K \times \RV_K^k$ is $(A \cup \RV_K)$-definable. Then there exists a finite (without loss non-empty) $A$-definable set $C \subset K$ such that, for all $x \in K$, the fiber $W_x \subseteq \RV_K^k$ only depends on the tuple
$(\rv(x - c))_{c \in C}$; i.e., if $\rv(x - c) = \rv(x' - c)$ for all $c \in C$, then $W_x = W_{x'}$.
\end{lem}

\begin{lem}[{\cite[Corollary~2.3.2]{CHRV-hmin2}}]\label{lem.hmin.mix}
Suppose that $K$ is a valued field of mixed characteristic with $0$-\hmix-minimal theory, that $A \subseteq K$ is an arbitrary (parameter) set, and that $W \subseteq K \times \RV_{|m'|,K}^k$ is $(A \cup \RV_{|m'|,K})$-definable, for some integer $m' \ge 1$. Then there exists a finite $A$-definable set $C \subset K$ and an integer $m \ge 1$ such that, for all $x \in K$, the fiber $W_x \subseteq \RV_{|m'|,K}^k$ only depends on the tuple
$(\rv_{|m|}(x - c))_{c \in C}$.
\end{lem}

\begin{rmk}
In \cite[Corollary~2.3.2]{CHRV-hmin2}, $W$ is required to be $A$-definable. However, if $\phi(x, a', a)$ defines $W \subseteq K \times \RV_{|m'|,K}^k$ for some $a' \in \RV_{|m'|}^\ell$ and $a \subset A$, then one can apply the original version of \cite[Corollary~2.3.2]{CHRV-hmin2} to the subset of $K \times \RV_{|m'|,K}^{k+\ell}$ defined by $\phi(x, y, a)$.
\end{rmk}

Note that if in Lemma~\ref{lem.hmin.mix}, one takes $K$ of equi-characteristic $0$, then $|m'| = |m| = 1$, so one just gets back Lemma~\ref{lem.hmin.equi}. In a similar way, 
equi-characteristic 0 and mixed characteristic could be treated together in this entire note.
However, since the latter is more technical, we will often explain the equi-characteristic 0 case first.

Any field $K$ as in Theorem~\ref{t.sphc.ex} is \emph{definably spherically complete}, i.e., every definable (with parameters) chain of balls $B_q \sseteq K$ has non-empty intersection, where $q$ runs over an arbitrary (maybe imaginary) definable set $Q$. In equi-characteristic $0$, this is just \cite[Lemma~2.7.1]{CHR-hmin}, and in general, it follows \emph{a posteriori} from Theorem~\ref{t.sphc.ex}. We nevertheless give a quick separate proof in the mixed characteristic case:

\begin{prop}\label{prop:dsc-mix}
If $K$ is a finitely ramified valued field of mixed characteristic with $0$-\hmix-minimal theory, then it is definably spherically complete.
\end{prop}

\begin{proof}
For each $\gamma \in \Gamma_K$, consider the unique open ball $B'_{<\gamma}$ of open radius $\gamma$ that contains some ball $B_q$, for some $q \in Q$, if it exists; let $Q' \subseteq \Gamma_K$ be the set of those $\gamma$ for which $B^\prime_{<\gamma}$ exists. Since the chains $(B_q)_{q \in Q}$ and $(B'_{<\gamma})_{\gamma \in Q'}$ have the same intersection, by working with the latter, we can, and do, assume without loss that $Q \subseteq \Gamma_K$ and that every $B_q$ is an open ball of radius $q$. 

Pick a finite set $C \subset K$ (using Lemma~\ref{lem.hmin.mix}) and an integer $m \ge 1$ such that for $x \in K$, the set $W_x := \{\xi \in \RV_K \mid x \in B_{|\xi|}\}$ only depends on the tuple $(\rv_{|m|}(x - c))_{c \in C}$. In particular, for every $q \in Q$,
if we have some $x' \in B_q$ and some $x \in K$ satisfying $\rv_{|m|}(x -  c) = \rv_{|m|}(x' - c)$ for every $c \in C$, then $x$ is also an element of $B_q$.

Suppose now that there exists a $q \in Q$ such that $B_q \cap C = \emptyset$. (Otherwise, $C$ contains an element of the intersection of all $B_q$ and we are done.)
We claim that every $q' \in Q$ satisfies $q' \ge q\cdot |m|$.
Since $K$ is finitely ramified, this claim implies that below $B_q$, the
chain is finite and hence has a minimum (which is then equal to the intersection of the entire chain).

To prove this claim, suppose for a contradiction that $q' < q\cdot |m|$. Pick any $x' \in B_{q'}$ and any $x \in K$ with $|x - x'| = q'$ (and hence $x \in B_q \setminus B_{q'}$). Then, for any $c \in C$ we have
$|c - x| \ge q$ (since $C \cap B_q = \emptyset$) and hence
$|c-x|\cdot|m| \ge q\cdot |m| > q' = |x - x'| = |(x -c) - (x'-c)|$.
Thus $\rv_{|m|}(x - c) = \rv_{|m|}(x' - c)$ for all $c \in C$, contradicting our choice of $C$ (and the fact that $x' \in B_{q'}$ and $x \notin B_{q'}$).
\end{proof}

Note that in mixed characteristic, the assumption of finite ramification is really necessary to obtain definable spherical completeness, as the following example shows.

\begin{exa}\label{exa:not-dsc}
Let $K$ be the algebraic closure of
$\mathbb{Q}$, considered as a valued field with the $p$-adic valuation. Fix any elements $a_n \in K$ ($n \in \NN_{\ge1}$) with $|a_n| = \lambda_n := |p|^{1-1/n}$ and set $a_I := \sum_{i \in I} a_i$ for $I \subset \NN_{\ge 1}$ finite. The balls $B_{\le\lambda_n}(a_I)$ ($n \in \NN_{\ge1}$, $I \subset \{1, \dots, n\}$) form an infinite binary tree (with $B_{\le\lambda_n}(a_I)$ containing
$B_{\le\lambda_{n+1}}(a_I)$ and $B_{\le\lambda_{n+1}}(a_{I\cup\{n\}})$), so since $K$ is countable, there exists a chain $B_n = B_{\le\lambda_n}(b_n)$ (where $b_n = a_{I_n}$ for suitable $I_n$) which has empty intersection. We fix such a chain.

Since $K$ is henselian, it is $\omega$-h$^{\mathrm{ecc}}$-minimal by \cite[Corollary~6.2.7]{CHR-hmin} and hence in particular 
$0$-\hmix-minimal (see the footnote on p.~\pageref{foo}).
Since each $B_n$ has a radius between $1$ and $|p|$, 
it is the preimage of a subset of the residue ring $\OO_K / B_{<|p|}(0) \subset \RV_{|p|}$, so we can turn the chain $(B_n)_{n \in \NN_{\ge 1}}$ into a definable family by expanding the language by a predicate on $\RV_{|p|}^2$ (e.g., $\{(\rv_{|p|}(c), \theta(n)) \mid n \in \NN_{\ge 1}, c \in B_n\}$,
for an arbitrary fixed injective map $\theta\colon \NN_{\ge 1} \to \RV_{|p|}$). By \cite[Proposition~2.6.5]{CHRV-hmin2},
$K$ stays $0$-\hmix-minimal when expanding the language by this predicate. However, it is not definably spherically complete in this expansion, as witnessed by the (now definable) chain $(B_n)_{n \in \NN_{\ge 1}}$.
\end{exa}

\section{Lemmas}

First we recall briefly a basic well-known lemma concerning the relationship between $\rv_{\lambda}$ and $\res_{\lambda}$, where $\res_{\lambda}\colon \OO_K \rightarrow \OO_K / B_{<\lambda}(0)$ is the canonical ring homomorphism onto the residue ring $\OO_K / B_{<\lambda}(0)$.

\begin{lem}
\label{lem.rv.res}
For all $a,a' \in \OO_K$ and all $\lambda \leq 1$ in $\Gamma^{\times}_K$,
\[
\rv_{\lambda}(a)=\rv_{\lambda}(a') \implies \res_{\lambda}(a)=\res_{\lambda}(a').
\]
\end{lem}
\begin{proof}
Recall that $\rv_{\lambda}(a)=\rv_{\lambda}(a')$ if and only if $|a-a'| < \lambda |a|$. Suppose that $a,a' \in \OO_K$ and $\rv_{\lambda}(a)=\rv_{\lambda}(a')$. As $|a| \leq 1$, we have that $|a-a'| < \lambda$. Therefore $a-a' \in B_{<\lambda}(0)$ so $\res_{\lambda}(a-a') = 0$. Hence, as $\res_{\lambda}$ is a ring homomorphism, we have $\res_{\lambda}(a) = \res_{\lambda}(a')$.
\end{proof}

\begin{lem}
\label{lem.finite}
Suppose $K$ is a valued field of characteristic 0. Let $A \sset K$ be a finite subset of $K$.
\begin{itemize}
 \item In case $K$ has finite residue characteristic $p$, let $\ell$ be any natural number such that $\#A < p^{\ell}$ and set $\lambda := |p^\ell|$;
 \item Otherwise, when $K$ has residue characteristic $0$, set $\lambda :=1$.
\end{itemize}
In either case, for any $a,a' \in A$,
\[a=a' \iff \{\rv_{\lambda}(a-v) : v \in A\} = \{\rv_{\lambda}(a'-v) : v \in A\}.\]
\end{lem}

\begin{proof}\renewcommand{\qedsymbol}{$\square$ (Lemma \ref{lem.finite})}
In case $A$ is empty or a singleton, the lemma is trivial. So we continue under the assumption that $\#A \geq 2$.

To each $a \in A$, associate the finite subset $D_a \sseteq \RV_{\lambda}$ defined by
\[
D_a := \{\rv_{\lambda}(a-v) : v \in A\},
\]
for $\lambda$ as in the statement of the lemma.

Suppose, towards a contradiction to the lemma, that there are $a_1 \neq a_2$ in $A$ such that $D_{a_1} = D_{a_2}$. Without loss of generality, we assume that $a_1-a_2 = 1$. This assumption can be made without loss as $\rv_{\lambda}$ is multiplicative; so dividing every element of $A$ by $a_1-a_2$ allows us to reduce to the case where $a_1-a_2 = 1$.

With this assumption we have that 
\[
\rv_{\lambda}(a_1 - a_2) = \res_{\lambda}(a_1 - a_2) = \res_{\lambda}(1) = 1.
\]
\begin{claim*}
Assuming the above, in particular that $a_1 \neq a_2$ but $D_{a_1}=D_{a_2}$, we find, for every natural number $i \ge 3$, an element $a_i$ in $A$,  satisfying
\begin{enumerate}
    \item $|a_i-a_1| \leq 1$ and
    \item  $\res_{\lambda}(a_1-a_i)=\res_{\lambda}(i-1)$.
\end{enumerate}
\end{claim*}
(Note that the condition already holds for $i = 1, 2$.)
\begin{proof}[Proof of the claim]\renewcommand{\qedsymbol}{$\square$ (Claim)}
Assume inductively that for some $i \in \NN$ we have found such an $a_i \in A$. By the definition of $D_{a_1}$ and the assumption that $D_{a_1} = D_{a_2}$, we have that 
\[
\rv_{\lambda}(a_1 - a_i) \in D_{a_1} = D_{a_2}.
\]
Therefore there exists some $a_{i+1} \in A$ such that $\rv_{\lambda}(a_2-a_{i+1}) = \rv_{\lambda}(a_1-a_{i})$.

First note that this implies $|a_2-a_{i+1}|=|a_1-a_{i}| \leq 1$; the last inequality by property in part 1 of the claim for $a_i$. Then, applying the ultrametric inequality, we calculate 
\[
|a_1-a_{i+1}| = |a_1 - a_2 + a_2 - a_{i+1}| \leq \max\{|a_1 - a_2|,|a_2 - a_{i+1}|\} \leq 1.
\]
This establishes that $a_{i+1}$ satisfies part 1 of the claim. 

Now by Lemma \ref{lem.rv.res}, the fact that $\rv_{\lambda}(a_2-a_{i+1}) = \rv_{\lambda}(a_1-a_{i})$ implies that $\res_{\lambda}(a_2-a_{i+1}) = \res_{\lambda}(a_1-a_{i})$. By the inductive assumption on $a_i$ we have $\res_{\lambda}(a_1-a_{i}) = i-1$, so $\res_{\lambda}(a_2-a_{i+1}) = i-1$. Now as $\res_{\lambda}(a_1-a_2) = 1$, we have
\[
\res_{\lambda}(a_1-a_{i+1})=\res_{\lambda}(a_1-a_2+a_2-a_{i+1})=\res_{\lambda}(a_1-a_2)+\res_{\lambda}(a_2-a_{i+1}) = i.
\]
This establishes that $a_{i+1}$ also satisfies part 2 of the claim. Hence, by induction, we have proved the claim.
\end{proof}

To finish the proof, it remains to verify, for some $N > \# A$, that all the elements \\
$\res_\lambda(0)$,~\dots,~$\res_\lambda(N-1)$ are distinct.  Since then, part 2 of the claim implies that $a_1, \dots, a_N$ are distinct, too, contradicting $N > \# A$; the contradiction going back to the assumption that $D_{a_1}=D_{a_2}$ while $a_1 \neq a_2$.

If $K$ is of residue characteristic 0 then we have set $\lambda = 1$, and clearly the elements of $\mathbb{N} \sseteq K$ have different residues in the residue field.
In case $K$ has finite residue characteristic $p$, recall that we chose $\ell$ so that $\#A < p^{\ell} =: N$ and that we have set $\lambda := |p^\ell|$. In particular, we have
$B_{<\lambda}(0) \sseteq p^{\ell} \OO_K$, so $0,1,...,p^{\ell}-1$ have different residues in the residue ring $\OO_K / B_{<\lambda}(0)$. This finishes the proof of the lemma.
\end{proof}

 \begin{defn}\label{defn.pcl}
Suppose that $K \subseteq M$ are valued fields.
An element $\alpha \in M$ is called a \emph{pcl} over $K$ (which stands for ``pseudo Cauchy limit'')
if for all $a \in K$ there exists an $a' \in K$ such that $|\alpha - a'| < |\alpha - a|$.
 \end{defn}

Note that if $\alpha$ is a pcl over $K$, then we can find an $a' \in K$ with $|\alpha - a'| < |\alpha - 0|$ and hence
$\rv(\alpha) = \rv(a')$. In particular, we have $\rv(\alpha) \in \RV_K$.

\begin{lem}\label{lem.emptydef}
Assume $K$ is an expansion of a valued field of characteristic 0 and either: 
\begin{itemize}
    \item $K$ has residue characteristic 0 and $\Th(K)$ is $0$-h-minimal; or
    \item $K$ is finitely ramified with finite residue characteristic and $\Th(K)$ is $0$-\hmix-minimal.
\end{itemize}
For $K$ of residue characteristic 0, set $\lambda := 1$; in case $K$ has finite residue characteristic $p$, let $\ell$ be any non-negative integer and set $\lambda := |p^\ell|$. 
Suppose that
\begin{itemize}
 \item $M$ is an elementary extension of $K$,
 \item $\alpha \in M$ is a pcl over $K$, and
 \item $(W_x)_{x\in M}$ is a $\emptyset$-definable family of subsets of $\RV_{\lambda,M}^n$ (for some $n$).
\end{itemize}
Then there exists some $a \in K$ such that $W_\alpha = W_a$. In particular, if the language contains constants for all elements of $K$, then $W_\alpha$ is $\emptyset$-definable.

\end{lem}

\begin{proof}\renewcommand{\qedsymbol}{$\square$ (Lemma \ref{lem.emptydef})}
First assume that $K$ is of equi-characteristic 0 and $\Th(K)=\Th(M)$ is $0$-h-minimal. By applying Lemma~\ref{lem.hmin.equi} in the valued field $M$ with the parameter set $A = \emptyset$, we obtain a finite $\emptyset$-definable subset $C$ of $M$ such that the fiber $W_x$ (for $x \in M$) only depends on the tuple $(\rv(x -c))_{c \in C}$. As $C$ is finite and $\emptyset$-definable, and as $K$ is an elementary submodel of $M$, we have $C \sset K$.

As $C \sset K$ and $\alpha$ is a pcl over $K$, for every $c \in C$, there exists $a' \in K$ such that $|\alpha - a'| < |\alpha - c|$. Hence as $C$ is finite, there exists some $a \in K$ such that for all $c \in C$ we have $\rv(a - c) = \rv(\alpha - c)$ . Hence by our choice of $C$, having the defining property given in Lemma \ref{lem.hmin.equi}, we have that $W_a = W_\alpha$ as required.

Suppose now that $K$ has finite residue characteristic $p$ and $\Th(K)=\Th(M)$ is 0-\hmix-minimal. Let $\lambda := |p^\ell|$ for some integer $\ell \geq 0$ as in the statement of the lemma. Now apply Lemma~\ref{lem.hmin.mix} in the valued field $M$ to the $\emptyset$-definable family $W_x$. This provides a finite $\emptyset$-definable set $C \sset M$ and the existence of an integer $m \geq 1$ such that the fiber $W_x $ (where $x \in M$) only depends on the tuple $(\rv_{|m|}(x -c))_{c \in C}$; so fix such an integer $m$. Again, as $C$ is finite and $\emptyset$-definable, and as $K$ is an elementary submodel of $M$, we have $C \sset K$.

Take any $c \in C$. As $\alpha$ is a pcl over $K$, for any positive integer $s$, there is a finite sequence $u_1,...,u_s$ of elements of $K$ such that $|\alpha - u_s| < \dots < |\alpha - u_1| < |\alpha - c|$.
Now since $M$ is finitely ramified, by taking $s$ to be a large enough positive integer and $\nu$ such that $|p^\nu|\le|m|$, we obtain that $|\alpha - u_s| < |p^\nu|\cdot|\alpha - c| \le |m|\cdot|\alpha - c|$. This implies that $\rv_{|m|}(\alpha - c) = \rv_{|m|}(u_s - c)$.
Since this works for each of the finitely many $c$ in $C$, we can find
an $a \in K$ such that
$\rv_{|m|}(\alpha - c) = \rv_{|m|}(a - c)$
for all $c \in C$. It then follows from the defining property of $C$ given by Lemma \ref{lem.hmin.mix} that
$W_a = W_\alpha$, as required.
\end{proof}

\section{Spherically complete models}

In this section, we prove the main result, Theorem~\ref{t.sphc.ex}. Throughout this section, we fix the valued field $K$ and work in an elementary extension $M \succ K$ (which we will at some point assume to be sufficiently saturated). Note that a field extension $L$ of $K$ is an immediate extension if and only if $\RV_{L} = \RV_K$.
 
We will deduce Theorem~\ref{t.sphc.ex} from the following proposition.
 
\begin{prop}\label{p.sphc.ex}
Suppose that $K \prec M$ are as in Theorem~\ref{t.sphc.ex}
(of equi-characteristic $0$ and $0$-h-minimal, or of mixed characteristic, finitely ramified, and $0$-\hmix-minimal).
Let $\alpha \in M$ be a pcl over $K$ (see Definition~\ref{defn.pcl}) and set $L := \acl_{\VF}(K, \alpha)$. Then we have $L \succ K$ and $\RV_{L} = \RV_K$.
\end{prop}

Here and in the following, $\acl_{\VF}$ means the field-sort part of the (relative model theoretic) algebraic closure inside the fixed extension $M$; we also use $\dcl_{\VF}$ in a similar way.

\begin{proof}[Proof of Proposition~\ref{p.sphc.ex}]\renewcommand{\qedsymbol}{$\square$ (Proposition \ref{p.sphc.ex})}
In the following, we work in the language with all elements from $K$ added as constants, so $\dcl_{\VF}(\emptyset) = K$
and $L = \acl_{\VF}(\alpha)$. Note that adding constants from the valued field to the language preserves
$0$-h-minimality and $0$-\hmix-minimality, by \cite[Theorem~4.1.19]{CHR-hmin} and \cite[Lemma~2.3.1]{CHRV-hmin2}.
(Alternatively, since we claimed that we only use Lemmas~\ref{lem.hmin.equi} and \ref{lem.hmin.mix} as definitions, note that the statements of these lemmas permit adding constants.)

We will prove the following (in this order):

\begin{claima*}
 $L = \dcl_{\VF}(\alpha)$.
\end{claima*}

\begin{claimb*}
 Suppose that $\mu \in \Gamma_K$ is either equal to $1$, or, if the residue characteristic of $K$ is $p$,
 $\mu=|p^\nu|$ for some $\nu \in \NN$.
 Then every subset of $(\RV_{\mu,M})^n$ (for any $n$) that is definable with parameters from $L$ is $\emptyset$-definable.
\end{claimb*}

\begin{claimc*}
 $L \prec M$.
\end{claimc*}

Then Claim B implies that $\RV_{L} = \RV_K$, since any element $\xi \in \RV_{L} \setminus \RV_K$ would provide a subset $\{\xi\}$ of  $\RV_{1,M}$ that is definable with parameters from $L$ but is not $\emptyset$-definable. Claim C implies that $K \prec L$. Thus upon proving the claims above,  we will be done.

\begin{proof}[Proof of Claim A]\renewcommand{\qedsymbol}{$\square$ (Claim A)}
Suppose that $\beta \in L$ and that the algebraicity of $\beta$ over $\{\alpha\}$ is witnessed by the formula $\phi(\alpha, y)$. Let $A$ be the finite set $\phi(\alpha,M)$, which is a subset of $L$ by the choice of $L$.
Depending on the residue characteristic of $K$, fix $\lambda$ as in the statement of Lemma \ref{lem.finite} (applied to this set $A$). Given any $b \in A$, we consider the (finite) set of ``$\RV_{\lambda}$ differences'' 
\[
D_b := \{\rv_\lambda(b - y) :  y \in \phi(\alpha, M)\}. 
\]
As $\phi(\alpha,M) \sseteq L$, each $D_b$ is a (finite) subset of $\RV_{\lambda,L}$.

By Lemma \ref{lem.finite}, different $b$ satisfying $\phi(\alpha, b)$ yield different finite sets $D_b \sseteq \RV_{\lambda,L}$. So $\beta$ can be defined by a formula (a priori over $L$) stating that
\[
\phi(\alpha, y) \,\,\,\wedge\,\,\, D_y = D_\beta,
\]
which uses the (finitely many) elements of $D_\beta \sseteq \RV_{\lambda,L}$ as additional parameters. To prove Claim A, it therefore suffices to check that $D_\beta \subseteq \RV_{\lambda,K}$ after all; from which it follows that $\beta$ can be defined over $K \cup \{\alpha\}$.

\begin{subclaim*}
 $D_\beta \subseteq \RV_{\lambda,K}$.
\end{subclaim*}

Now, with $x$ running over $M$, consider the $\emptyset$-definable family of sets 
\[
W_x := \{\rv_{\lambda}(y - y') : y \in \phi(x, M) \wedge y' \in \phi(x, M) \}. 
\]
Each such $W_x$ is a subset of $\RV_{\lambda,M}$. Note that $D_\beta \subseteq W_\alpha$, which is clear from their respective definitions. Now  we apply Lemma \ref{lem.emptydef} to the $\emptyset$-definable family $W_x$ to obtain that $W_\alpha$ is $\emptyset$-definable. As $\acl_{\VF}(\emptyset) = K$ (and $W_\alpha$ is finite), we therefore have that $W_\alpha \subseteq \RV_{\lambda,K}$. So in particular $D_\beta \subseteq \RV_{\lambda,K}$, which establishes the subclaim. 
\end{proof}

\begin{proof}[Proof of Claim B]\renewcommand{\qedsymbol}{$\square$ (Claim B)}
By Claim A, any set definable over $L$ is in fact $\{\alpha\}$-definable. So take any subset $W_\alpha$ of $(\RV_{\mu,M})^n$ that is definable over $L$ and let $\psi(\alpha,y)$ be a definition for it, where $\psi(x,y)$ is a formula over $\emptyset$.  Now consider the $\emptyset$-definable family of sets $W_x := \psi(x,M)$ in $M$. Applying Lemma~\ref{lem.emptydef} to this family yields that, for some $a \in K$, we have $W_\alpha=W_a$, and hence is $\emptyset$-definable, as required.
\end{proof}

\begin{proof}[Proof of Claim C]\renewcommand{\qedsymbol}{$\square$ (Claim C)}
We need to verify that every non-empty subset $Y \subseteq M$ that is definable with parameters from $L$ already contains a point of $L$.

From the $0$-h-minimality (or $0$-\hmix-minimality assumption), we obtain that there exists a 
finite $L$-definable set $C = \{c_1, \dots, c_r\} \subset M$ such that $Y$ is a union of fibers of the map 
\[
\begin{array}{rcl}
 \rho\colon  M &\to& (\RV_{\mu,M})^r,\\ 
   y &\mapsto& (\rv_{\mu}(y - c_i))_{i\leq r}
\end{array}
\]
for a suitable $\mu$. Indeed, if $K$ is of equi-characteristic 0
and $\Th(K)=\Th(M)$ is 0-h-minimal, then Lemma~\ref{lem.hmin.equi} (applied to $Y$ considered as a subset of $M \times \RV^0_M$) provides such a $C$ with $\mu = 1$. If $K$ is of mixed characteristic
and $\Th(K) = \Th(M)$ is $0$-\hmix-minimal, then we use Lemma~\ref{lem.hmin.mix} instead, which yields the analogous statement with $\mu = |p^\nu|$ for some integer $\nu \ge 1$.

Fix $C$ and $\mu$ for the rest of the proof; also fix an enumeration of $C$ and the corresponding map $\rho$ as above.
Since $L= \acl_{\VF}(L)$, we have $C \subseteq L$, so $\rho$ is $L$-definable, and so is the image $\rho(Y) \sseteq (\RV_{\mu,M})^r$ of $Y$. 
Hence by Claim B, that image is $\emptyset$-definable. By assumption $Y \ne \emptyset$, so we also have $\rho(Y) \ne \emptyset$; as $K$ is an elementary substructure of $M$, $\rho(Y) \cap (\RV_{\mu,K})^r$ is non-empty, too.

Choose any $\xi = (\xi_1, \dots, \xi_r) \in \rho(Y) \cap (\RV_{\mu,K})^r$ in this intersection. Since $Y$ is a union of fibers of $\rho$, to show that $L \cap Y$ is non-empty,
it suffices to prove that the preimage of $\xi$ in $M$,
\[
\rho^{-1}(\xi) = \bigcap_{i \le r} (c_i + \rv_{\mu}^{-1}(\xi_i)),
\]
has non-empty intersection with $L$.
Each of the finitely many sets $B_i := c_i + \rv_{\mu}^{-1}(\xi_i)$ is a ball and the intersection of all of them is
non-empty, since $\xi \in \rho(Y) \subseteq \rho(M)$. By the ultrametric inequality, this intersection is equal to one of those balls, say $B_j$. Since $\xi_j \in \RV_{\mu,K}$, there exists $a \in \rv_{\mu}^{-1}(\xi_i) \cap K$. Thus we obtain the desired element $c_j + a \in L \cap B_j = L \cap \rho^{-1}(\xi) \subseteq L \cap Y$.
\end{proof}

As explained after the statements of the claims, the proposition follows.
\end{proof}

We now conclude with the proof of Theorem~\ref{t.sphc.ex}.

\begin{proof}[Proof of Theorem~\ref{t.sphc.ex}]\renewcommand{\qedsymbol}{$\square$ (Theorem~\ref{t.sphc.ex})}
Fix some $(\#K)^+$-saturated elementary extension $M \succ K$. 
By Zorn's Lemma, there exist maximal elementary extensions $N \succ K$ satisfying $N \prec M$ and such that $\RV_{N} = \RV_K$. Fix $N$ to be one of them. The rest of the proof consists in showing that such an $N$ is spherically complete (using Proposition~\ref{p.sphc.ex}).

Suppose, towards a contradiction, that $(B_i)_{i \in I}$ is a nested family of closed balls in $N$ such that $\bigcap_{i \in I} B_i = \emptyset$. Being nested, this family of balls defines a partial type over $N$. Without loss, we assume that all the $B_i$ have different radii, so there are at most as many balls as the cardinality of the value group of $N$. As $N$ is an immediate extension of $K$, the value group of $N$ is the same as the value group of $K$. Therefore $M$ is sufficiently saturated to contain a realisation $\alpha \in M$ of this partial type. This $\alpha$ is a pcl over $N$, since for every $a \in N$ there exists an $i \in I$ with $a \notin B_i$; hence for any $a' \in B_i$, we have $|\alpha - a'| < |\alpha - a|$. Now Proposition~\ref{p.sphc.ex} (applied to $N \prec M$) implies that $L:=\acl_{\VF}(N,\alpha)$ is a proper elementary extension of $N$ in $M$ satisfying $\RV_{L} =\RV_{N} = \RV_K$, contradicting the maximality of $N$. We conclude instead that $L:=N$ is spherically complete.
\end{proof}

As a further model theoretic addendum, we add that the spherically complete elementary extension $L$ obtained above is also at least as saturated as $K$.

\begin{rmk}\label{r.sph.sat}
Suppose $K$ satisfies the assumptions of Theorem~\ref{t.sphc.ex} and is $\kappa$-saturated. 
Then any immediate spherically complete elementary extension $N \succ K$ is also $\kappa$-saturated.
\end{rmk}

\begin{proof}
We write the proof in the mixed-characteristic case, which is easily simplified to the equi-characteristic 0 context. For ease of notation we write $\RV_\bullet := \bigcup_{n\in \NN} \RV_{|n|}$. Note that by our assumption of finite ramification, $\RV_K = \RV_N$ implies $\RV_{\bullet,K} = \RV_{\bullet,N}$.
Indeed, this follows from the existence of natural short exact sequences of multiplicative groups $k^\times \hookrightarrow (\RV_{\lambda',K} \setminus \{0\}) \twoheadrightarrow (\RV_{\lambda,K} \setminus \{0\})$ when $\lambda'$ is a predecessor of $\lambda$ in the value group, and where $k$ is the residue field of $K$.

Fix a $\kappa^+$-saturated and strongly $\kappa^+$-homogeneous $M \succ N$. Suppose we are given a type $p$ over a set $E \sseteq N$ of parameters of cardinality less than $\kappa$. Without loss, $E = \acl_{\VF}(E)$. Choose a realization $\varepsilon \in M$ of $p$.

Fix some enumeration of $E \times \NN$ and let $\xi := \left(\rv_{|n|}(\varepsilon - e)\right)_{e \in E, n\in \NN} \in \RV_{\bullet,M}^{\kappa'}$ (for some $\kappa' < \kappa$) be the sequence of leading term differences from $\varepsilon$ to $E$.
Let $q:=\tp(\xi/E)$ be its type. By stable embeddedness of $\RV_{\bullet}$, each formula $\psi \in q$ is equivalent to a formula $\psi'$ with parameters from $\RV_{\bullet,N} = \RV_{\bullet,K}$, so we find a set $E' \subset \RV_{\bullet,K}$ of cardinality $\le \# E < \kappa$ such that $q' := \tp(\xi/E')$ implies $\tp(\xi/E)$.
Since $K$ is $\kappa$-saturated, $q'$ is realized in $\RV_{\bullet,K}$. Pick such a realization $\xi'$, let $\sigma \in \operatorname{Aut}(M/E)$ be an automorphism sending $\xi$ to $\xi'$, and set $\varepsilon' := \sigma(\varepsilon)$.
Note that we have $\rv_{|n|}(\varepsilon' - e) \in \RV_{\bullet,K}$ for every $e \in E$ and every $n \in \NN$.

Now by $0$-\hmix-minimality, our type $p = \tp(\varepsilon'/E)$ is implied by $\{\rv_{|n|}(x - e) = \rv_{|n|}(\varepsilon' - e) : e \in E \wedge n \in \NN \}$.
Indeed, for each formula in $p$, the set $C$ provided by Lemma~\ref{lem.hmin.mix} lies in $E = \acl_{\VF}(E)$.
The formulas $\rv_{|n|}(x - e) = \rv_{|n|}(\varepsilon' - e)$ define a chain of nested balls, so they have non-empty intersection in $N$. Any point in this intersection is a realization of $p$.
\end{proof}

\section{Related questions}
\label{sec:qu}

Any power-bounded $T$-convex valued field $K$ in the sense of L. van den Dries and A. H. Lewenberg \cite{DL.Tcon1} is $1$-h-minimal by \cite[Theorem~6.3.4]{CHR-hmin} and hence also $0$-h-minimal, so Theorem~\ref{t.sphc.ex} applies.
However, this case is already given by E. Kaplan \cite[Corollary 1.12]{Kap2021}, and Kaplan's result establishes uniqueness (up to the relevant notion of isomorphism over $K$) of the spherically complete elementary extension $L$ in that case.
We can ask whether uniqueness holds more generally:

\begin{qu}
Assuming equi-characteristic $0$, or mixed characteristic and finitely ramified,
is the spherically complete immediate extension $L$ of $K$ constructed in Theorem~\ref{t.sphc.ex} unique as a structure in the expanded language (up to isomorphism over $K$)?
\end{qu}

It is then natural to consider what happens in the $T$-convex case when $T$ is not power-bounded. (Such $K$ are not $0$-h-minimal; see \cite[Remark~6.3.5]{CHR-hmin}.) By \cite[Theorem 2]{KKS1997}, any valued real closed field with an exponential does not have any spherically complete elementary extension (\cite[Remark 7]{KKS1997}). So, as explained in \cite[Remark 1.13]{Kap2021}, if $K$ is a $T$-convex valued field arising from a non power-bounded o-minimal $T$, then no spherically complete elementary extension (or even model) exists. One might suspect that such $K$ are not even definably spherically complete, but we do not know whether this is the case.

\medskip

Recall that in mixed characteristic, the assumption of finite ramification seemed to be natural since then, $0$-\hmix-minimality implies definable spherical completeness. However, one can still ask slightly more generally:

\begin{qu}
\label{qu.not.finram}
If $K$ is $0$-\hmix-minimal and definably spherically complete, does it have an (immediate) spherically complete elementary extension?
\end{qu}

Question~\ref{Q} also has a positive answer in some non-$0$-h-minimal cases. Concretely, in \cite[Theorem 6.3]{Kap2021}, conditions are given for certain $T$-convex valued fields expanded by certain derivations to have spherically complete immediate elementary extensions. Note that the field of constants of a non-trivial derivation on $K$ is a definable subset for which Lemma~\ref{lem.hmin.equi} does not hold, so these expansions are typically not $0$-h-minimal.

\bibliography{Canstratbib}
\bibliographystyle{alpha}

\end{document}